\documentclass{amsart}

\usepackage{vmargin,amsmath,amsfonts,amssymb,amsthm,xcolor,amscd,latexsym,xspace,enumerate}
\usepackage[pagebackref, colorlinks=true,linkcolor=blue,citecolor=red]{hyperref}
\usepackage{hyperref}
\usepackage{pdfpages}
\usepackage{tikz-cd}
\usetikzlibrary{arrows}
\usetikzlibrary{intersections}
\usepackage{pgfplots}
\usetikzlibrary{calc,3d,shapes, pgfplots.external, intersections}
\usepackage{pgfplots}

\newtheorem{theorem}{Theorem}[section]
\newtheorem{corollary}[theorem]{Corollary}

\newtheorem{proposition}[theorem]{Proposition}
\newtheorem{conjecture}[theorem]{Conjecture}
\theoremstyle{definition}
\newtheorem{definition}[theorem]{Definition}

\newtheorem{problem}[theorem]{Open Problem}
\newtheorem{example}[theorem]{Example}
\numberwithin{equation}{section}

\begin{document}

\title[On the global breadth of finite groups with nontrivial partitions]{On the global breadth of finite groups\\ 
with nontrivial partitions}

\author[S.K. Muhie]{Seid Kassaw Muhie}
\address{Seid Kassaw Muhie \endgraf  
Institute of Data Science and Digital Technologies\endgraf 
 Vilnius University\endgraf
 Akademijos str. 4, LT-08663, Vilnius, Lithuania} 

\author[D.E. Otera]{Daniele Ettore Otera}
\address{Daniele Ettore Otera \endgraf  
Institute of Data Science and Digital Technologies\endgraf 
 Vilnius University\endgraf
 Akademijos str. 4, LT-08663, Vilnius, Lithuania} 

\author[F.G. Russo]{Francesco G. Russo}
\address{Francesco G. Russo \endgraf
School of Science and Technology\endgraf
University of Camerino\endgraf
via Madonna delle Carceri 9, Camerino, Italy\endgraf
and\endgraf
Department of Mathematics and Applied Mathematics\endgraf
University of the Western Cape\endgraf
Private Bag X17, 7535,  Bellville, South Africa}



\begin{abstract} In a series of recent contributions on the notion of global breadth $\mathbf{B}(G)$ of a finite group $G$, it was interesting to observe  the structural conditions arising from the classification of finite groups of $\mathbf{B}(G)=8$. This  motivated the study of a new class of finite groups, namely  $\mathcal{H}=\{G \ | \ G \ \mbox{satisfies the condition } \ |G| \le \mathbf{B}(G)(\mathbf{B}(G) + 1)\}$ and very little is known about $\mathcal{H}$. Here we focus on the groups with nontrivial partitions (according to the terminology of Baer, Kegel and Kontorovich), determining first that  $\mathbf{B}(G)$ is achieved via the local breadth in connection with the order of maximal cyclic subgroups. Then  we show that $\mathcal{H}$ contains projective special linear groups, projective general linear groups and Suzuki groups, supporting the conjecture that all finite groups with nontrivial partitions belong to $\mathcal{H}$. The presence of large families of simple groups in $\mathcal{H}$ is shown  for the first time here.
\end{abstract}

\maketitle

\section{Introduction and statement of the main results}

The present paper deals with finite groups only. For all divisors $k$ of the order $|G|$ of a group $G$, Frobenius \cite[Satz 1]{frobenius}  investigated   the cardinality  of the set \begin{equation}\label{lkg} L_k(G) = \{x \in G \ | \ x^k=1\},
\end{equation} noting that $|L_k(G)|$ is  divisible by $k$.  Successive contributions were made by Iiyori and Yamaki  \cite{yamaki1, yamaki2, yamaki3, yamaki4, yamaki5} in the context of simple groups, so that the original  Frobenius' Problem (of studying systematically the arithmetic properties of \eqref{lkg}) was addressed.

On the other hand, can we get structural information on $G$ from bounds on \eqref{lkg} ? In a certain sense, this is the Inverse Frobenius' Problem, and a series of investigations have been recently made  in \cite{ hr1, hr2, hr3, ms, meng2016, meng2017,  jm, russo}. With the symbol $\mathrm{Div}(m)$, we denote those natural numbers which are divisors of a positive integer $m$, and with $\pi(m)$ the  prime divisors of $m$. Note that $\mathrm{Div}(G)$  denotes the set of all divisors of $|G|$ and $\pi(G)$ the set of all prime divisors of $|G|$. The positive integer
\begin{equation}\label{lb}\mathbf{b}_k(G) = \frac{|L_k(G)|}{k}
,\end{equation}
is called \textit{local breadths} of $G$ in \cite{hr1, hr2, hr3}, and the  optimal value
\begin{equation}\label{gb} 
\mathbf{B}(G)=\max \{ \mathbf{b}_k(G) \ | \ k \in \mathrm{Div}(G) \},
\end{equation}
is called  \textit{global breadth}  of $G$  in \cite{hr1, hr2, hr3}. If $\exp(G)$ is the exponent of $G$,  one has from these definitions that $\mathbf{b}_h(G) \le \mathbf{b}_k(G)$ when $k=\gcd(h,\exp(G))$, hence for the determination of $\mathbf{B}(G)$ only integers $k \in \mathrm{Div}(\exp(G))$ are relevant. Of course for $k=\exp(G)$ we have  \begin{equation}
\mathbf{b}_{\exp(G)}(G)= \frac{|L_{\exp(G)}(G)|}{\exp(G)} = \left( \frac{1}{\exp(G)}\right)  \   \cdot \  |L_{\exp(G)}(G)|
\end{equation} and so the local breadth of $G$ agrees with   $|L_{\exp(G)}(G)|$  up to the multiplicative normalization factor  $1/\exp(G)$; in other words there are no differences to measure  $|L_{\exp(G)}(G)|$ up to the exponent of $G$, or, to find the local breadth of $G$, when $k=\exp(G)$.  

There are alternatives for the  computation of the local breadth, looking at \cite{gp, russo} and at other contributions which focus on the number of cyclic subgroups of a given order   $m$ 
\begin{equation}\label{phi} 
\mathbf{c}_ m= |\{ C\leq G \ |  \ C \ \mbox{ is a cyclic subgroup of order} \ m \}|
\end{equation}
In fact if $m$ is a divisor of $k$,  the  Euler function $\phi(m) $ is connected with $\mathbf{b}_k(G)$ and $\mathbf{c}_m$ by
\begin{equation}\label{phib}
|L_k(G)| = \underset{m|k}{\sum}{\mathbf{c}_m  \ \phi(m)} = k \ \mathbf{b}_k(G).
\end{equation}

Several reasons motivate to investigate  local and global breadth:

\medskip
\medskip
\medskip

-- groups $G$ of $\mathbf{B}(G) \le 2$ are classified in terms of metacyclic groups; 

\medskip

-- groups $G$ of $\mathbf{B}(G)=3$ are also  classified structurally; 

\medskip

-- the cases of $\mathbf{B}(G)=4$ and $ \textbf{B}(G)\leq 7$ imply that $G$ is solvable;

\medskip

-- for groups of large $\mathbf{B}(G)$  new classes of groups appeared in the literature\footnote{ These are the so-called \textit{refined groups} in \cite[Definition 1.1]{hr3}, the \textit{deduced groups} in \cite[Definition 1.2]{hr3} and  the $\mathcal{Q}$-$groups$ in  \cite[Section 2 ]{hr2}. Classifications are available in \cite{hr1, hr2, hr3, russo}.}.  

\medskip
\medskip
\medskip

The main idea is to give answers to questions of the following type:

\medskip
\medskip

\textbf{Generalized Inverse Frobenius' Problem.} Classify all groups $G$ such that $|L_k(G)| \le f(k)$ for all $k \in \mathrm{Div}(\exp(G))$, where $f : k \in  \mathrm{Div}(G) \subseteq \mathbb{N} \mapsto f(k) \in \mathbb{N}$ is a  function depending only on  $k$.

\medskip
\medskip

The Generalized Inverse Frobenius' Problem has been studied only in special situations, namely when $f(k)$ is a constant function, or a linear function, or a quadratic function. Corresponding classifications of groups have been described in \cite{hr1, hr2, hr3, ms, meng2016, meng2017}.

\begin{problem} For arbitrary expressions of $f(k)$ (i.e.: polynomials or else), the Generalized Inverse Frobenius' Problem is still open. In particular, the following class of groups
\begin{equation}\label{classh}\mathcal{H}=\{G \ | \ G \ \mbox{satisfies the condition } \ |G| \le \mathbf{B}(G)(\mathbf{B}(G) + 1)\}
\end{equation}
was introduced in \cite[Section 8]{hr3} in connection with local breadth and global breadth, and seem to contain groups with relevant structural restrictions.\end{problem}

In order to study the aforementioned problem, we recall a classical notion which allow us to look at groups from the point of view of the covering theory in topology.

\begin{definition}[Groups with Nontrivial Partitions, see \cite{baer1}]
A collection $\mathrm{P}(G)$ of nontrivial subgroups of (a nontrivial group) $G$  is said to be \textit{a partition of G} if every nontrivial element of $G$ belongs to a unique subgroup in $\mathrm{P}(G)$. We say that  $\mathrm{P}(G)$ is a \textit{trivial partition} for $G$ if $|\mathrm{P}(G)|=1$. \end{definition}

Connections between the structure of the subgroups lattice $\mathrm{L}(G)$  of a group $G$ and the group $G$ itself  have  been largely explored in the literature. Of course, $\mathrm{P}(G)$ is a subset of $\mathrm{L}(G)$  and  groups with nontrivial partitions appear in contributions of  Baer \cite{baer1, baer2}, Kegel \cite{kegel1, kegel2, kegel3}, Kontorovitch \cite{kont} and Suzuki \cite{suzuki}.  Our first main result allows us to determine the global breadth of $G$ is found at the order of one of its maximal cyclic subgroups.

\begin{theorem}\label{l0} If $G$ is a group with nontrivial partition, then there exists a maximal cyclic subgroup $M$ of $G$ such that $\mathbf{B}(G)=\mathbf{b}_{|M|}(G).$
\end{theorem} 

Our second main result  is devoted to determine that the two dimensional projective special linear groups $\mathrm{PSL}(2,q)$, the two dimensional projective general linear groups $\mathrm{PGL}(2,q)$ and the Suzuki groups $\mathrm{Sz}(2^{2n+1})$ belong to $\mathcal{H}$ (as usual $\mathbb{F}_q$ denotes the field of prime power order $q$). The presence of groups with nontrivial partitions in $\mathcal{H}$ is observed here for the first time. 

\begin{theorem}\label{maintheorem} There are groups with nontrivial partitions in $\mathcal{H}$, namely
\begin{itemize}
\item[{\rm (i)}.] The symmetric group $S_4$ on four elements belongs to $\mathcal{H}$.   
    \item[{\rm(ii)}.] If $q=3^n$ for $n\ge 2$ or $q$ is a prime power for $q\ge 5$ , then there is a constant $c \ge 2 \sqrt{2}$ such that
$$|\mathrm{PSL}(2,q)| \le  c  \ \mathbf{B}(\mathrm{PSL}(2,q)) \ \sqrt{\mathbf{B}(\mathrm{PSL}(2,q))}.$$ In particular, 
$\mathrm{PSL}(2,q) \in  \mathcal{H}$ for all $q \ge 4$. 
    \item[{\rm(iii).}] If $q=2^{2n+1}$ for all $n\ge 1$, then there is a constant $c \ge 2 \sqrt{2}$ such that
$$|\mathrm{Sz}(q)| \le  c  \ \mathbf{B}(\mathrm{Sz}(q)) \ \sqrt{\mathbf{B}(\mathrm{Sz}(q))},$$
and  $\mathrm{Sz}(q) \in  \mathcal{H}$ for all $q \ge 8$. 
\item[{\rm (iv).}] If $q$ is a prime power for $q\ge 4$, then there is a constant $c \ge 2 \sqrt{2}$ such that
$$|\mathrm{PGL}(2,q)| \le  c  \ \mathbf{B}(\mathrm{PGL}(2,q)) \ \sqrt{\mathbf{B}(\mathrm{PGL}(2,q))}$$
  and  $\mathrm{PGL}(2,q) \in  \mathcal{H}$ for all $q \ge 4$.
\end{itemize}

\end{theorem}

Among simple groups with a nontrivial partition,  there are very few simple groups (see \cite [Corollary 3.5.12] {rs}), therefore one could explore whether the presence of a partition is really necessary among simple groups, in order to find large families of simple groups in $\mathcal{H}$.

The paper is organised in four sections: Section 1 formulates the main results and the motivations to be interested in these. Some preliminary notions are explained properly in Section 2, where previous results on the global breadth are recalled from the literature. Section 3 is devoted to prove our first main result and its consequences, while Section 4 is devoted to prove our second main result and ends with two examples which illustrate the computational simplifications which we are offering via Theorems \ref{l0} and \ref{maintheorem}. A final   open conjecture on $\mathcal{H}$ is placed in Section 4, resembling some evidences which we encountered in our research. Notations and terminology are standard and follow \cite{huppert, rs}.

\section{Preliminaries and results on groups with nontrivial partitions}

We begin to recall that the  \textit{holomorph} $\mathrm{Hol}(G)$ of a group $G$ (see \cite{huppert}) is the extension of $G$ by its automorpshim group and that \begin{equation}\mathrm{Aff}(\mathbb{F}_q)= \{ \phi : x \in \mathbb{F}_q \mapsto ax + b \in \mathbb{F}_q  \ | \ a \in \mathbb{F}^\times_q, b \in \mathbb{F}_q  \} 
\end{equation} is the  group of affine mappings from  $\mathbb{F}_q$ onto itself, see \cite[Kapitel V, Beispiele 8.6]{huppert}. Here  $\mathbb{F}^\times_q= \mathbb{F}_q \setminus\{0\}$ is the multiplicative group of  $\mathbb{F}_q$ and the general linear group $\mathrm{GL}(2,q)$ of dimension two on $\mathbb{F}_q$ is always a subgroup of $\mathrm{Aff}(\mathbb{F}_q)$, since $\mathrm{Aff}(\mathbb{F}_q)=T(\mathbb{F}_q) \rtimes \mathrm{GL}(2,q)$, that is, it turns out to be  semidirect product of $\mathrm{GL}(2,q)$ by the group $T(\mathbb{F}_q)$ of the translations from $\mathbb{F}_q$ to $\mathbb{F}_q$. 
We also recall that given a prime $p$ and a group $G$, \begin{equation}H_p(G)=\langle g \in G \mid x^p \neq 1\rangle
\end{equation} is the $Hughes$ $subgroup$ of  $G$, see \cite{kegel2}. Of course, $H_p(G)$ turns out to be the smallest subgroup of $G$ outside of which all elements of $G$ have order $p$ and, if $G$ has $\exp(G)=p$, then $H_p(G)=1$. Note also that $H_p(G)$ is always characterstic in $G$. Having in mind the results in \cite[Kapitel V, \S 8]{huppert}, we shall also mention that a group $G$ is  $Hughes$-$Thompson$ $type$, if it is not a $p$-group and $H_p(G) \neq G$. Kegel \cite{kegel2} has shown that groups of Hughes-Thompson type have  $H_p(G)$ which is nilpotent of index $|G : H_p(G)| = p$. The notion of $Frobenius$ $group$ is well known, see \cite[Kapitel V, \S 8, Definition 8.1]{huppert},  and that of $Suzuki$ $group$ $\mathrm{Sz}(2^{2n+1})$ as well, see \cite[Kapitel II \S 10, Bemerkung 10.15]{huppert}. Then details are omitted here, concerning their formal definitions.

Groups with partitions were largely studied and classified in \cite{baer1, baer2, kegel1, kegel2, kegel3, kont, suzuki}. They have relevant properties of structural nature and probabilistic counting formulas for their number of commuting subgroups have been recently found in \cite{seidnew}.

\begin{theorem}[See \cite{mf},  Classification Theorem, pp.119-120]\label{niceclassification}Let $G$ be a  group with
 a nontrivial partition. Then $G$ is isomorphic with exactly one of the following groups:
\begin{itemize}
\item [(i).] the symmetric group $S_4$ on four elements;
\item [(ii).] a $p$-group with $H_p(G) \neq G$;
\item [(iii).] a group of Hughes-Thompson type;
\item [(iv).] a Frobenius group;
\item [(v).]  $\mathrm{PSL}(2,q)$ for $q \ge 4  $;
\item [(vi).]  $\mathrm{PGL}(2,q)$ for $q \ge 5$ odd prime power;
\item [(vii).] $\mathrm{Sz}(2^{2n+1})$.
\end{itemize}
\end{theorem}

The following notion is of more recent introduction:

\begin{definition}[See \cite{hr2}]\label{refined} A group $G$ is said to be \textit{refined} if we have  $\mathbf{B}(G) \neq \mathbf{B}(G/N )$ for all proper normal subgroups $N$ of $G$. 
\end{definition}

Essentially, refined groups are those groups possessing global breadth always different from their nontrivial homomorphic images. Moreover refined groups may be of arbitrary global breadth: they have in fact  a  sophisticated structure (below $C_m$ is cyclic group of order $m$ for $m$ positive integer).

\begin{theorem}[See \cite{hr2}, Theorem 3.1] \label{anidea}
Let  $G$ be a refined group of $\mathbf{ B}(G) = m \ge 2$. Then the following is true:
\begin{itemize}
\item[(i).]  $|G|$ is not divisible by primes $p > 2m-1$.
\item[(ii).] If a prime $p$ with $m < p \le 2m-1$ divides $|G|$, then $G$ is a subgroup of  $\mathrm{Hol}(C_p)$  and $p-m \in  \mathrm{Div} (m-1)$.
\item[(iii).] If $m$ is prime and $m \in \mathrm{Div}(G)$, then one of the following conditions  is satisfied:  
\begin{itemize}
\item[(1).] $G \simeq C_m \times C_m$; 
\item[(2).] $G \simeq C_m \times U$, where $U$ is a nonabelian subgroup of $\mathrm{Hol}(C_m)$;
\item[(3).] $m+1$ is a prime power and $G \simeq \mathrm{Aff}(\mathbb{F}_{m+1})$.    
\end{itemize}
\item[(iv).] If $G$ is abelian, then $ \exp(G) \in \mathrm{Div}(m)$. Furthermore there exists an abelian subgroup $R$ of $G$ such that $G/R \simeq C_t \times C_t$, where $t = \exp(G)$.
\end{itemize}
\end{theorem}

Of course, simple groups are refined, due to the absence of nontrivial proper normal subgroups. Hence Theorem \ref{anidea} gives an idea of what can happen for linear groups which are simple groups. Another result, which is very useful to recall, deals with the computation of the global breadth in case of extensions.

\begin{proposition}[See \cite{hr2}, Corollary 4.1]\label{tek1}
If $s = p_1^{n_1}...p_r^{n_r}$ is a factorization in the  product of prime powers of the integer $s \ge 1$ such that $p_i^{n_i} \equiv 1 \mod k$ for all $i=1,2,\ldots, r$,  there is an abelian group $N$ of order $s$ and a 
cyclic subgroup $W \subseteq \mathrm{Aut}(N)$ such that $|W| = k$ and nontrivial elements of $W$ operate without fixed points on $N$. The extension $G$ of $N$ by $W$ satisfies the relation \begin{equation}k \cdot \mathbf{ B}(G) = (k-1)s + 1;
\end{equation} in particular
 $2 \cdot \mathbf{ B}(G) > s$.\end{proposition}

We end this section with a few numerical results, which are reported for the convenience of the reader from \cite{hr1, hr2, hr3, ms, meng2016, meng2017}, collecting the relevant information. Below \begin{equation}D_{2n}=\langle a,t \mid a^n=t^2=1, \ t^{-1}at=a^{-1}\rangle
\end{equation} denotes the usual dihedral group of order $2n$, where $n \ge 1$. Since  $D_2 \simeq C_2$ and $D_4 \simeq C_2 \times C_2$, it is clear that relevant nonabelian examples begin from $n \ge 3$. Note also that alternating groups on $n$ elements are denoted by $A_n$,  symmetric groups on $n$ elements by $S_n$ and the quaternion group of order 8 by $Q_8$.

\medskip
\medskip

\textbf{Refined nonabelian groups of global breadth at most 11}\\
\\
$\mathbf{B}(G)$ \ \ \ \  \ \ \ \  $G$ refined nonabelian 

\medskip
\medskip

2 \ \ \ \ \ \ \ \  \ \ $D_6$. 

\medskip
\medskip

3 \ \ \ \ \ \ \ \  \ \ $A_4$, $D_{10}$, $D_8$, $D_6 \times C_3$.

\medskip
\medskip

4 \ \ \ \ \ \ \ \  \ \ $D_{14}$, $D_{12}$, $\mathrm{Hol}(C_5)$, $A_4 \times C_2$, $ \mathrm{SL}(2,3) \times C_2$, $Q_8 \times C_2$, 

\medskip

 \ \  \ \ \ \ \ \ \ \  \ \  $\langle a,b  \ | \ a^4=b^4=a^{-1}bab=1 \rangle$. 

\medskip
\medskip

5 \ \ \ \ \ \ \  \ \ \ $D_{18}$, $D_{16}$, ${(\mathrm{Aff}(\mathbb{F}_9))}^4$, $(\mathrm{Hol}(C_7))^2$, $D_{10} \times C_5$,                $\mathrm{Hol}(C_5) \times C_5$, $S_4$. 
                     
 \medskip
\medskip

6 \ \ \ \ \ \ \ \ \ \  $D_{22}$, $D_{20}$, $\mathrm{Hol}(C_7)$, $D_8 \times C_2$, 

\medskip

\ \ \ \ \ \ \ \ \ \ \ \ $\langle x,y,z \ | \ x^3=y^9=z^2= (yz)^2= [x,z]= y^5 z y z=1 \rangle$.

\medskip
\medskip

7 \ \ \ \ \ \ \ \ \ \ $D_{26}$, $D_{24}$, $\mathrm{Aff}(\mathbb{F}_8)$, $D_{14} \times C_7$, $(\mathrm{Hol}(C_7))^2 \times C_7$, $\mathrm{GL}(2,3)$, $\mathrm{Hol}(C_7) \times C_7$. 

\medskip
\medskip

8 \ \ \ \ \ \ \ \ \ \ See Theorem \ref{classificationofbg8} below.

\medskip
\medskip

9 \ \ \ \ \ \ \ \ \ \  $D_{34}$, $D_{32}$,  $(\mathrm{Hol}(C_{13}))^4$, $(\mathrm{Hol}(C_{11}))^2$, $A_4 \times C_3$, $D_6 \times C^2_3 $, $C_3 \times \mathrm{Aff}(\mathbb{F}_9)$, 

\medskip

\ \ \ \ \ \ \ \ \  \ \ \ \ $D_8 \times C^2_3$, $C_3 \times (\mathrm{Aff}(\mathbb{F}_9))^2$, $C_3 \times (\mathrm{Aff}(\mathbb{F}_9))^4$, 

\medskip

\ \ \ \ \ \ \ \ \  \ \ \ \  $K_0=C_3 \ltimes C^2_3 =\langle x,y,z \ | \  x^3 = y^3 = z^3= [x,y]=[x,z] = 1, yz = zyx \rangle$,  

\medskip

\ \ \ \ \ \ \ \ \ \ \  \ \ $K_1=C_2 \ltimes K_0$ with $Z(K_1)=C_3$,   $K_2=C_2 \ltimes K_0 $ with $Z(K_2)=1$,

\medskip

\ \ \ \ \ \ \ \ \ \ \  \ \ $K_3= C_4 \ltimes K_0$ with $Z(K_3)=C_3$. 

\medskip
\medskip

10 \ \  \ \ \ \ \ \ \ $D_{38}$, $D_{36}$, $(\mathrm{Hol}(C_{13}))^3$, $\mathrm{Hol}(C_{11})$, $D_{16} \times C_2$, $D_6 \times C^2_5$, $(\mathrm{Hol}(C_7)^2 \times C^2_2$. 

\medskip
\medskip

11 \ \ \ \ \ \ \ \ \ $D_{42}$, $D_{40}$, ${(\mathrm{Hol}(C_{13}))}^2$, ${(\mathrm{Aff}(\mathbb{F}_{16})}^5$, $C_{11} \times D_{22}$,

\medskip

\ \ \ \ \ \ \ \ \ \ \  \ \                       $ C_{11} \times {(\mathrm{Hol}(C_{11}))}^2$, $C_{11} \times \mathrm{Hol}(C_{11})$,

\medskip

\ \ \ \ \ \ \ \ \ \ \  \ \                 $ \langle a,b,c \ | \ a^3 = b^4 = c^4 = [b,c] = 1, a^{-1}ba = c,  a^{-1}ca = b^{-1}c^{-1}\rangle $.

\medskip
\medskip

As mentioned in the list above, the case of global breadth $8$ is challenging. It requires different techniques, and new ideas, so we cannot list simply the groups as made for all the other cases, but we need of a result of structure of independent interest.

\begin{theorem}[Refined Groups of Global Breadth 8, see \cite{hr3}] \label{classificationofbg8}

Let $G$ be a refined group of $\mathbf{B}(G) = 8$.

\begin{itemize}
\item[(i).] If $G$ is nilpotent, then $G$ is a $2$-group.

\item[(ii).] If $G$ possesses a noncyclic Sylow subgroup $S$ of odd order, then $|S| = 9$ and either $G \simeq \mathrm{Aff}(\mathbb{F}_9)$ or  $G \simeq D_6 \times D_6$.

\item[(iii).] If $5 \in \pi(G)$, then either $G \simeq D_{30}$, or $G \simeq \mathrm{Hol}(C_5) \times C_2$, or  $G \simeq A_5$.

\item[(iv).] If $7 \in \pi(G)$, then either $G \simeq D_{14} \times C_2$ or $G \simeq \mathrm{Aff}(\mathbb{F}_8) \times C_2$.

\item[(v).] If $\pi(G)=\{2,3\}$ and $G$ is nonabelian with  cyclic Sylow $3$-subgroup $S$ and with Sylow $2$-subgroup $T$, then $|S|=3$ and one of the following conditions is satisfied:
\begin{itemize}
\item[(1)]  $S$ is  normal and either  $T \simeq C_2 \times C_2 \times C_2$, or $T \simeq C_4 \times C_4$ ;
\item[(2)]$T$ is normal and either  $T \simeq C_2 \times C_2 \times C_2 \times C_2$, or  $T \simeq Q_8 \times C_4$;
\item[(3)] neither $S$ nor $T$  are  normal and $G=HK$ with $H \simeq A_4$,  $K \simeq C_4$ and $G/Z(G) \simeq S_4$.
\end{itemize}

\item[(vi).] There is no $G$ such that $p \in \pi(G)$ and $p \ge 11$.
\end{itemize}

\end{theorem}

Of course, Proposition \ref{tek1} does not help in case of simple groups, but neverthless it is useful in case we have subgroups of simple groups, which can be obtained as appropriate extensions.

\section{Proof of the First Main Theorem}

We begin to prove of our first main result, where the presence of appropriate cyclic maximal subgroups is somehow guaranteed by the simultaneous presence of nontrivial partitions. 

\begin{proof}[Proof of Theorem \ref{l0}]
Let $d\in \mathrm{Div}(G)$ and  assume that $d=p$ is prime.  According to the First Sylow Theorem, $G$ has a $p$-subgroup $H_i$ of size $p$ and  there are exactly $n_p=kp+1$ of these $p$-subgroups by the Third Sylow Theorem (for some $k \ge 1$). The existence of a nontrivial partition for $G$ implies \begin{equation} \label{npp} |L_p(G)|= \Big|\bigcup_{i=0}^{n_p}H_i \Big| = |G: N_G(H_i)| \ (p-1)+1= n_p(p-1)+1= \end{equation} 
\[=(kp+1) (p-1)+1= kp^2 - kp + p =  kp^2 +  (1-k) p .\]

\medskip
Next, suppose that $d=p^s$ for some $s \ge 2$ and that $|G|=p^sr$ such that $\gcd(p^s,r)=1$. Again by Sylow Theorems, we may consider  a Sylow $p$-subgroup $H$ of $G$ and that there are  $n' _p=kp+1$ Sylow $p$-subgroups. 
hence we have (again by the presence of nontrivial partitions)
\begin{equation} \label{nppp}|L_d(G)| \le \Big|\bigcup_{g \in G}H^g \Big| \le (kp+1)|H|=(kp+1)p^s \end{equation} and the equality is achieved if and only if $s=1$, i.e.: when  $H$ is cyclic of order $p$.

If $M$ is a maximal cyclic subgroup of $G$ of order a power of $p$, then by Sylow Theorems $M$ should be contained in some Sylow $p$-subgroup of $G$. If $M$ has order different from a power of $p$, then we may find $M$ in more than one Sylow subgroup of $G$. Therefore we shall examine different cases on the basis of  $|M|$.

\begin{itemize}
\item [Case 1.] Assume that all maximal cyclic subgroups of $G$ have prime order. Then $G$ does not contain an element of order a prime power $d=p^t$ for $2 \le t \le n$ and $p \in \pi(G)$. The elements of the sets $L_p(G)$ and $L_{p^2}(G)$ are therefore identical. Thus \begin{equation} \label{00} 
 L_p(G)=L_{p^2}(G)= \cdots = L_{p^t}(G)= \cdots= L_{p^n}(G). \end{equation}
 In particular, for $p^2, p^3$ and so on we have by \eqref{npp} that 
\begin{equation} \mathbf{b}_{p^2}(G)=\frac{|L_{p^2}(G|}{p^2}   = \frac{kp + (1-k)}{p}  < kp+1-k = \frac{|L_p(G)|}{p}=\mathbf{b}_{p}(G) \end{equation}  and so on $\ldots < \mathbf{b}_{p^3}(G) < \mathbf{b}_{p^2}(G)  < \mathbf{b}_p(G)$. 
Additionally $G$ does not contain  elements of order  $d=pq$ for distinct primes $p,q \in \pi(G)$, then it is easy to check that  
\begin{equation} \label{eqpppp}
   L_{d}(G)\supseteq L_p(G) \ \ \mbox{and} \ \  L_{d}(G)\supseteq L_q(G).
   \end{equation}
 However, we may consider the interval $[C_{pq}/1]$ of the subgroups lattice
 $\mathrm{L}(G)$ (see details in \cite[\S 1.1, pp.5--10]{rs}) and using \eqref{phib} we have $\mathbf{c}_{pq}=0$ and
 \begin{equation} \label{eqpppp2}
   \mathbf{b}_{pq}(G)=\frac{L_{pq}(G)}{pq}={\underset{m | pq} \sum}  \mathbf{c}_m\phi(m) = \frac{1+\mathbf{c}_p(p-1)+\mathbf{c}_q(q-1)}{pq} \end{equation}
   $$= \frac{1+\mathbf{c}_p(p-1)}{pq}+ \frac{\left(1+\mathbf{c}_q(q-1)\right)-1}{pq}
   =
   \frac{\mathbf{b}_{p}(G)}{q}+\frac{\mathbf{b}_{q}(G)}{p}-\frac{1}{pq}$$ $$
 \ \ \Longrightarrow \ \ \mathbf{b}_p(G) \geq \mathbf{b}_{d}(G) \ \ \mbox{and} \ \  \mathbf{b}_q(G) \geq \mathbf{b}_{d}(G). $$
The same argument is  valid when $d$ is the product of more than two distinct primes, say $d=pqr$ with $\gcd(p,q,r)=1$, and so on. Therefore $\mathbf{B}(G)$ is reached by $\mathbf{b}_{|M|} (G)$  for some maximal subgroup $M$ of prime order.

\medskip
\medskip

 \item[Case 2.] Assume that  some maximal cyclic subgroup $M$ of $G$ is  either isomorphic to a cyclic subgroup $C_{p^t}$ of order $p^t$ for some $t \ge 1$, or  is isomorphic to $C_k =C_{p_1^{t_1}} \times  C_{p_2^{t_2}} \times \ldots \times C_{p_r^{t_r}}$, where  $k=p_1^{t_1} p_2^{t_2} \ldots p_r^{t_r}$ is a factorization of $k \ge 1$ in prime powers such that $\gcd(p_i, p_j)=1$ for all $i,j \in \{1, 2, \ldots, r\}$ and $r \ge 1$.   Then one can check that both the lattice of all subgroups $\mathrm{L}(C_{p^t})$ of $M \simeq C_{p^t}$  and that  $\mathrm{L}(C_{k})$  of $M \simeq C_k$ are distributive sublattices of $\mathrm{L}(G)$, where in the first case we have also a chain. It is  evident from counting the elements of $G$, that $G$ does not contain any elements of order $d > p^t$ in case $|M|=p^t$ and $d > k$ when $|M|=k$. This indicates that $L_d(G)=L_{p^t}(G)$ and  $L_{d}(G)= L_{k}(G). $ 
 Arguing as in Case 1, we get
\begin{equation} \mathbf{b}_d(G)=\frac{|L_{p^t}(G|}{d}<\frac{|L_{p^t}(G)|}{p^t}=\mathbf{b}_{p^t}(G),   \ \mbox{and}  \ \ 
\mathbf{b}_d(G)=\frac{|L_{k}(G|}{d}< \frac{|L_k(G)|}{k}=\mathbf{b}_{k}(G).
\end{equation} 
Lastly, it remains to check $\mathbf{b}_{d}(G)$ when
 $d < p^t$  for $M\simeq C_{p^t}$ and $d < k$ for $M\simeq C_{k}$.  Now consider $M\simeq C_{p^t}$ and remember that we are in a chain where we may consider  the interval $[C_{p^t}/1]$  and using \eqref{phib} we have $\mathbf{c}_p=\mathbf{c}_{p^2}=\cdots =\mathbf{c}_{p^t}=\mathbf{c}_d.$ This means that the ratio of numbers of elements of $G$  of order exactly  $d=p^i$ divided by  $\phi(p^i)$ is constant for every divisor  of $p^t$. Therefore, 
\begin{equation}
\mathbf{b}_{p^t}(G)=\frac{\underset{d | p^t}{\sum} \mathbf{c}_d \phi(d)}{p^t} = \frac{1}{p^t} \underset{d | p^t}{\sum}  \ \mathbf{c}_d  \ \phi(d) = \frac{1}{p^t} \left(1+\underset{d \neq 1} {\underset{d|p^t}{\sum}} \mathbf{c}_d  \ \phi(d)\right)= \frac{1}{p^t} \left(1+\mathbf{c}_d   \ {\overset{t}{\underset{i=1} \sum}} \phi(p^i)\right)\end{equation}  
\[= \frac{\mathbf{b}_{p}(G)}{p^{t-1}} + \frac{\mathbf{c}_d}{p^t} \left( \underset{1 \neq p\neq d} {\underset{d | p^t} \sum}  \phi(d) \right)=\frac{\mathbf{b}_{p^{t-1}}(G)}{p}+\mathbf{b}_p(G)-\frac{1}{p}.\]
From what we wrote above we may conclude that  $\mathbf{b}_{p^t}(G) \geq \mathbf{b}_{p}(G)$, but actually
$\mathbf{b}_{p^t}(G)\geq \mathbf{b}_{p^{t-1}}(G) \geq \cdots \geq\mathbf{b}_{p}(G).$ We end with the case of $M \simeq C_{p^t}$ and it remains to consider the case of $M \simeq C_k$, but the same argument can be applied, since we may consider the interval $[C_k/1]$ and divisors of the form $d=p_i^{t_i}$ of $k$. Note that  $\mathrm{L}(C_{k})$ is a complete sublattice (but generally is not a chain) and so we get to $\mathbf{b}_{k}(G) \geq \mathbf{b}_{p_i^{t_i}}(G).$ 
Thus, if $M \simeq C_{p^t}$ or $M \simeq C_k$, $\mathbf{b}_{d}(G)$ may attain its maximum if either $d=p^t$ or $d=k$, assuming it does not occur at the cyclic subgroup of order prime. 
\end{itemize} 
Since Case 1 and Case 2 exhaust all the possible options, the result follows.
\end{proof}

We now  apply the arguments of the proof of Theorem \ref{l0} to some relevant groups, in order to have a partial version of  Theorem \ref{maintheorem}.

\begin{corollary}\label{l1}For $q=3^n$ with $n \ge 2$ or for all odd prime powers $q \ge 5$
 \begin{equation} \label{pslodd}
 \mathbf{B}(\mathrm{PSL}(2,q)) + 2 = \frac{q(q-1)}{2},
 \end{equation} but in case $q=2^m$ (with $m \ge 2$)
 \begin{equation}\mathbf{B}(\mathrm{PSL}(2,q)) = 2^{2m-1}.
 \end{equation}
\end{corollary}

\begin{proof}First of all, we know that $|\mathrm{PSL}(2,2^m)|=2^m \ (2^{2m}-1)$ and $|\mathrm{PSL}(2,q)|=q(q^2-1)/\gcd(2,q-1)$. Counting the involutions in $\mathrm{PSL}(2,2^m)$ we get \begin{equation}|L_2(\mathrm{PSL}(2,2^m))|=|\{x \in \mathrm{PSL}(2,2^m) \ | \ x^2=1\}|=2^{2m}.
\end{equation} This means that \begin{equation}\mathbf{b}_2(\mathrm{PSL}(2,2^m))=2^{2m-1}.
\end{equation}
Counting the elements of order $(q-1)/2$ in $\mathrm{PSL}(2,q)$ when $q$ is odd, we get
\begin{equation}|L_{\frac{q-1}{2}}(\mathrm{PSL}(2,q))|=|\{x \in \mathrm{PSL}(2,q) \ | \ x^{\frac{q-1}{2}}=1\}|=\frac{q^3-2q^2-3q+4}{4}
\end{equation}
and so \begin{equation}\mathbf{b}_\frac{q-1}{2}(\mathrm{PSL}(2,q))|=\frac{q(q-1)}{2}-2.
\end{equation}
We want to give more details on the nature of these computations, showing that the argument of the proof of Theorem \ref{l0} may be specialized in a significant way here. For the rest of this proof, let's denote by $q$ a power of a prime, which can be eventually $2$. We should check that among $\mathbf{b}_e(\mathrm{PSL}(2,q))$ with $e \in \mathrm{Div}(\mathrm{PSL}(2,q))$ the value $\mathbf{b}_2(\mathrm{PSL}(2,q))$ indeed realizes the global breadth $\mathbf{B}(\mathrm{PSL}(2,q))$ when $q$ is even, and the above value $\mathbf{b}_\frac{q-1}{2}(\mathrm{PSL}(2,q))$ indeed realizes the global breadth $\mathbf{B}(\mathrm{PSL}(2,q))$ when $q$ is odd.  This is going to support the relevance of the role of the cyclic maximal subgroups in these computations.

Of course $\mathrm{PSL}(2,q)$ is not simple group when $q=2$ and $q=3$. So, let $q=3^n$ for $n \ge 2$ or $q$ be odd prime powers, with $q \ge 5$ and in case $q=2^m$, with $m \ge 2$.  

Assume $e=p \in \pi(\mathrm{PSL}(2,q))$. Then, we look at the number of Sylow subgroups in $\mathrm{PSL}(2,q)$ and repeat the argument of the first part of the proof Theorem \ref{l0}. By the First Sylow Theorem $\mathrm{PSL}(2,q)$ has $p$-subgroup $P_i$ of size $p$ and with $n_p$  number of these $p$-subgroups we are able to find that $L_p(\mathrm{PSL}(2,q))$ has $|L_p(\mathrm{PSL}(2,q))|-1$ nontrivial elements of order $p$, each $P_i$ has $p-1$ nontrivial elements of order $p$ and $P_i\cap P_j=\{1\}$ for all $i\ne j$, by the presence of a nontrivial partition for $\mathrm{PSL}(2,q)$.  Then 
\begin{equation} \label{np}
|L_p(\mathrm{PSL}(2,q))|= \Big|\bigcup_{i=0}^{n_p}P_i \Big|= n_p(p-1)+1= |G : N_G(P_i)| \ (p-1)+1=kp^2+kp+p+1.
\end{equation}

Arguing again as in the first part of the proof of Theorem \ref{l0}, we now look at  $e=p^s$ with $|\mathrm{PSL}(2,q)|=p^sr$ such that $\gcd(p^s,r)=1$. Let $P$ is a Sylow $p$-subgroup of $\mathrm{PSL}(2,q)$ and it is not normal. In \cite [Example 3.5.1]{rs}, it is established that the set of maximal cyclic subgroups is a partition of $\mathrm{PSL}(2,q))$. Then using Theorem \ref{l0}, it is more appropriate to focus on the set of maximal cyclic subgroups rather than the Sylow $p$-subgroups of $\mathrm{PSL}(2,q)$, when calculating the size of $ L_e(\mathrm{PSL}(2,q))$, which can be achieved by applying \eqref{phib}.

First of all consider $q=2^m$.  Note that $\mathrm{PSL}(2,q)$ does not have an element of order $e=2^r$ if $2\le r \leq n$. Then we still have a chain of equalities which we discussed in the argument of the proof of Theorem \ref{l0}, namely
$L_2(\mathrm{PSL}(2,q))=L_{4}(\mathrm{PSL}(2,q))= \ldots = L_{2^t}(\mathrm{PSL}(2,q))= \ldots= L_{2^n}(\mathrm{PSL}(2,q))$ and in particular for $e=4$ 
\begin{equation}
\mathbf{b}_e(\mathrm{PSL}(2,q))=\frac{|L_e(\mathrm{PSL}(2,q))|}{4} < \frac{|L_e(\mathrm{PSL}(2,q))|}{2}=\mathbf{b}_{p}(\mathrm{PSL}(2,q)).
\end{equation}
However, if $q=2^m$, one can check that $\mathrm{PSL}(2,q)$ has an element of order $e$ if $e=2$. This means that there is a maximal cyclic $2$-subgroup  $Q$ contained in the Sylow $2$-subgroup $P$ of $\mathrm{PSL}(2,q)$.  Therefore, for $e=2^s$,   we have
\begin{equation}
\mathbf{b}_e(\mathrm{PSL}(2,q))=\frac{|L_e(\mathrm{PSL}(2,q))|}{2^t}=\frac{|L_2(\mathrm{PSL}(2,q))|}{2^t} < \frac{|L_2(\mathrm{PSL}(2,q))|}{2}=\mathbf{b}_{2}(\mathrm{PSL}(2,q)).
\end{equation}

Again, since the maximal cyclic subgroups of $\mathrm{PSL}(2,q)$ has order $2$, $\frac{2^m-1}{2}$ and $\frac{2^m+1}{2}$ for $q=2^m$, then, it is easy to check that for $e=p_1p_2$ with $e\in \mathrm{Div}(\frac{2^m-1}{2})$ or $e\in \mathrm{Div}(\frac{2^m+1}{2})$ we get   
\begin{equation} \label{eqcompa2}
     L_{p_1}(\mathrm{PSL}(2,q))\subseteq L_{e}(\mathrm{PSL}(2,q)) \ \ \mbox{and} \ \  L_{p_2}(\mathrm{PSL}(2,q))\subseteq L_{e}(\mathrm{PSL}(2,q)) 
\end{equation} and so
\begin{equation}
\mathbf{b}_{p_1}(\mathrm{PSL}(2,q)) \leq \mathbf{b}_{e}(\mathrm{PSL}(2,q)) \ \ \mbox{and} \ \  \mathbf{b}_{p_2}(\mathrm{PSL}(2,q)) \leq \mathbf{b}_{e}(\mathrm{PSL}(2,q)). 
\end{equation}
In particular, if $e=2^m-1$  and $e=2^m+1$, then we have
\begin{equation}
    \mathbf{b}_{2^m-1}(\mathrm{PSL}(2,q))= 2^{2m-1}-1\leq 2^{2m-1}
\ \ \mathrm{and} \ \ 
    \mathbf{b}_{2^m+1}(\mathrm{PSL}(2,q))= 2^{2m-1}-(2^m+1)\leq 2^{2m-1}.
\end{equation}
Note that $(\mathrm{PSL}(2,q)$ does not have also an element of order $e\ge 2^{m}-1$ and $e\ge 2^m+1$, and the above inequality holds true by applying Theorem \ref{l0}.
Therefore it is enough to check only $\mathbf{b}_e(\mathrm{PSL}(2,q))$ at $e=2$, $e=\frac{2^m-1}{2}$ and $e=\frac{2^m+1}{2}$ to obtain $\mathbf{B}(\mathrm{PSL}(2,2^m))$. 
Therefore, $\mathbf{b}_e(\mathrm{PSL}(2,q))$  reaches the maximum when $e=2$, in case $q=2^m$.

Finally, we adapt the final step of the proof of Theorem \ref{l0} also to the present situation. Now we assume $q=p^n$ for the odd prime $p$ with $n\geq 1$ and, in the case $p=3$,  $n\geq 2$. Similarly, $\mathrm{PSL}(2,q)$ does not have an element of order $e=p^t$ for any odd prime $p$ if $2\leq t \leq n$. This implies,
\begin{equation}
\mathbf{b}_e(\mathrm{PSL}(2,q))=\frac{|L_e(\mathrm{PSL}(2,q))|}{p^t} < \frac{|L_e(\mathrm{PSL}(2,q))|}{p}=\mathbf{b}_{p}(\mathrm{PSL}(2,q)) \ \ \mbox{for} \ \ 2\leq t\leq n.
\end{equation}
with $\mathbf{b}_{p}(\mathrm{PSL}(2,q))=q.$    In addition, $\mathrm{PSL}(2,q)$ has an element of order $e$ for all $e\in \mathrm{Div}((q-1)/2)$ and $e\in \mathrm{Div}((q+1)/2)$.  With the same argument of the previous case, it is important to compute $\mathbf{b}_{e}(\mathrm{PSL}(2,q))$ at $e=\frac{q-1}{2}$ and $e=\frac{q+1}{2}$. Thus, for $q\geq 5$
\begin{equation}
    \mathbf{b}_{\frac{q+1}{2}}(\mathrm{PSL}(2,q))= \frac{q^2-2(q+1)}{2}\leq  \frac{q^2-q-4}{2}=\mathbf{b}_{\frac{q-1}{2}}(\mathrm{PSL}(2,q)). 
\end{equation} 
To end the proof, it remains to compute only $\mathbf{b}_e(\mathrm{PSL}(2,q))$ if $e\geq q-1$ and $e\ge q+1$, with $e\in \mathrm{Div}(\mathrm{PSL}(2,q))$. In this case  $(\mathrm{PSL}(2,q)$ has no element of order $e$, hence
\begin{equation}
    \mathbf{b}_e(\mathrm{PSL}(2,q)) \leq 2^{2m-1}.
\end{equation} 
Therefore, $\mathbf{b}_e(\mathrm{PSL}(2,q))$  reaches the maximum when $e=\frac{q-1}{2}$, in case $q$ is a prime power except $q=3$.  The result follows completely.  
\end{proof}

The remaining two corollaries of the present section illustrate that the argument of the proof of Theorem \ref{l0} can be specialized (not only to the case of $\mathrm{PSL}(2,q)$ as we have just done) to the groups with nontrivial partitions in Theorem \ref{niceclassification} (vi) and (vii). We sketch the main ideas.

\begin{corollary}\label{l2} Let $q=2^{2n+1}$ for $n \ge 1$. Then,
\begin{equation}\mathbf{B}(\mathrm{Sz}(2^{2n+1}))= \frac{q^4-q^3-2(q+1)}{2}.
 \end{equation} 
\end{corollary}

\begin{proof} Arguing as in the proof of Theorem \ref{l0}, we consider $q=2^{2n+1}$ for $n \ge 0$ and know that $|\mathrm{Sz}(q)|= q^2(q^2+1)(q-1)$. Counting the elements of order $q-1$ in $\mathrm{Sz}(q)$, we get
\begin{equation}
|L_{q-1}(\mathrm{Sz}(q))|=|\{x \in \mathrm{Sz}(q) \ | \ x^{q-1}=1\}|=(q-1) \ \ \left(\frac{q^4-q^3-2(q-1)}{2}\right)
\end{equation}
and so \begin{equation}\mathbf{b}_{q-1}(\mathrm{Sz}(q))=\frac{q^4-q^3-2(q+1)}{2}.
\end{equation}
Of course if $n=0$, then $\mathrm{Sz}(2)$ is a group of order 20, which is nonsimple. Assume $q=2^{2n+1}$ for $n \ge 1$ and we should check that among $\mathbf{b}_c(\mathrm{Sz}(q))$ with $c\in \mathrm{Div}(\mathrm{Sz}(q))$ the value $\mathbf{b}_{q-1}(\mathrm{Sz}(q))$ indeed realizes the global breadth $\mathbf{B}(\mathrm{Sz}(q))$. We follow the same procedure of the proof of Corollary \ref{l1}, omitting the details. We just report the following computations, which are useful to identify the main local breadths that one needs. In fact we  only need to check the additional values $\mathbf{b}_{c}(\mathrm{Sz}(q))$ when $c=4$, $c=q+2^{n+1} + 1$ and $c=q-2^{n+1} + 1$, that is, 
\begin{equation}
    \mathbf{b}_4(\mathrm{Sz}(q))=  2^{6n+4}, \ \ \ \ \ \ 
    \mathbf{b}_{q-2^{n+1}+1}(\mathrm{Sz}(q))=  2^{6n+4}+2^{4n+3}+2^{2n+1}+2^{n+1}+1,
\end{equation}  
\begin{equation}
\mathbf{b}_{q+2^{n+1} + 1}(\mathrm{Sz}(q))= 2^{6n+3}+2^{2n+1}-2^{n+1}+1.
\end{equation}
Therefore we get what we claimed for $q=2^{2n+1}$  with $n \geq 1$, namely \begin{equation}\mathbf{B}(\mathrm{Sz}(q)) = \frac{q^4-q^3-2(q+1)}{2}=\:\left(2^{2n+1}\right)^4-\left(2^{2n+1}\right)^3-2\left(2^{2n+1}+1\right).\end{equation} 
\end{proof}

\begin{corollary}\label{l3} For all odd prime powers $q \ge 5$, we get
 \begin{equation}\mathbf{B}(\mathrm{PGL}(2,q)) = \frac{q^2+q-2}{2}, 
 \end{equation} 
and if $q=2^m$ for $m\ge2$, then
 \begin{equation} \label{pgl2}
\mathbf{B}(\mathrm{PGL}(2,q)) = 2^{2m-1}. \end{equation}
\end{corollary}

\begin{proof} Again we overlap the same argument of the proof of Theorem \ref{l0}, sketching the main ideas and omitting some details on the computations. Since $|\mathrm{PGL}(2,2^m)|=q(q-1)(q+1)$, counting the elements of order $q$ and $q-1$ in $\mathrm{PGL}(2,q)$, we obtain for  $q=2^m$ and $m\ge2$ that 
\begin{equation}
|L_2(\mathrm{PGL}(2,q))| =|\{x \in \mathrm{PGL}(2,q) \  | \ x^2=1\}| = {q^{2}} \end{equation} and so
\begin{equation} 
\mathbf{B}(\mathrm{PGL}(2,q)) = \frac{q^2}{2} =2^{2m-1}.
\end{equation}
When $q$ is an odd prime bigger than five, we get
\begin{equation} 
|L_{q-1}(\mathrm{PGL}(2,q))|=|\{x \in \mathrm{PGL}(2,q) \ | \ x^{q-1}=1\}|=\frac{(q-1)(q^2+q-2)}{2}
\end{equation}
and so \begin{equation} 
\mathbf{b}_{q-1}(\mathrm{PGL}(2,q))=\frac{q^2+q-2}{2}.
\end{equation}
 The main idea is always the same, that is,  we should check that among all possible values of local breadth $\mathbf{b}_f(\mathrm{PGL}(2,q))$ with $f \in \mathrm{Div}(\mathrm{PGL}(2,q))$, the global breadth $\mathbf{B}(\mathrm{PGL}(2,q))$ is reached by the value $\mathbf{b}_{q-1}(\mathrm{PGL}(2,q))$  when $q\geq 3$ is odd and by the value $\mathbf{b}_{2}(\mathrm{PGL}(2,q))$  when $q=2^m$ with $m\ge2$. Computational details are omitted (since the logic has been widely discussed).
\end{proof}

\section{Proof of the Second Main Theorem}

We have now all that we need, in order to prove our second main result and may upper bound the size of most of the groups which appear in Theorem \ref{niceclassification}, determining when they belong to $\mathcal{H}$.

\begin{proof}[Proof of Theorem \ref{maintheorem}]\

(i). We note by direct computations (or by the list of Section 2 above) that   \begin{equation}\mathbf{B}(S_4)=5 \ \ \mathrm{and} \ \ |S_4| \le \mathbf{B}(S_4) (\mathbf{B}(S_4) +1) = 30 \ \ \Longrightarrow \ \ S_4 \in \mathcal{H}
\end{equation}  and  $S_4$ also has a nontrivial partition. This case follows. 

\medskip
\medskip
\medskip
\medskip
\medskip
\medskip
\medskip
\medskip
(ii). Assume that $q$ is even, that is, $q=2^m$ for some $m \ge 2$. We claim that
\begin{equation}|\mathrm{PSL}(2,2^m)|^2  \le  8  \ \mathbf{B}(\mathrm{PSL}(2,2^m)^3.
\end{equation}
In fact for all $m \ge 1$
\begin{equation}2^{2m+1} \ge 1   \ \  \Longrightarrow \  \ 2^{4m} \ge 2^{4m}-2^{2m+1}+1=(2^{2m}-1)^2  
\end{equation}
and so we may conclude by Corollary \ref{l1} that
\begin{equation}\label{extra1}|\mathrm{PSL}(2,2^m)|^2 =  2^{2m} \ \cdot \  (2^{2m}-1)^2 \le 2^{2m}  \ \cdot \ 2^{4m}
\end{equation}
$$=2^{2m} \ \cdot \  8 \ \cdot \ 2^{4m} \  \cdot \ 2^{-3}=8 \ \cdot \ 2^{6m-3} = 8 \ \cdot \ \mathbf{B}(\mathrm{PSL}(2,2^m)^3.$$
Making the square root to both sides of the above inequality,  the claim follows.

In order to show the same for an odd prime power $q \ge 5$ or for $q=3^n$ with $n \ge 2$, we note that
\begin{equation}8 \ \mathbf{B}(\mathrm{PSL}(2,q))^3 = 8 \  \left(\frac{q(q-1)}{2} -2\right)^3=  (q^2-q-4)^3
\end{equation}
and that
\begin{equation}|\mathrm{PSL}(2,q)|^2= \frac{q^2 \ (q^2-1)^2}{\gcd(2,q-1)^2}= \frac{q^6-2q^4+q^2}{\gcd(2,q-1)^2}.  \end{equation}
This allows us to conclude that the condition
\begin{equation} \gcd(2,q-1)^2 \  (q^2-q-4)^3  \ge  (q^6-2q^4+q^2)
\end{equation} 
is always satisfied and so
\begin{equation}|\mathrm{PSL}(2,q)|^2  \le  8  \ \mathbf{B}(\mathrm{PSL}(2,q))^3,
\end{equation}
from which the result follows. 

Finally, $\mathrm{PSL}(2,q)$ possesses a nontrivial partition when $q \ge 4$ by Theorem \ref{niceclassification} and is a simple group when $q \ge 4$. Therefore it remains to check that $\mathrm{PSL}(2,q) \in \mathcal{H}$ when $q \ge 4$.   This is true, because for all $q \ge 4$ we have by Corollary \ref{l1} and by $c^2 \ge 8 $ that
\begin{equation} 0 \le c^2 -2   \le  {\mathbf{B}(\mathrm{PSL}(2,q))} + \frac{1}{\mathbf{B}(\mathrm{PSL}(2,q))}
\end{equation}   
\begin{equation}
\Longrightarrow  \  c^2 \  {\mathbf{B}(\mathrm{PSL}(2,q))}  - 2  {\mathbf{B}(\mathrm{PSL}(2,q))} \le  {\mathbf{B}(\mathrm{PSL}(2,q))}^2  +1  
\end{equation}   
\begin{equation}
   \Longrightarrow  \  c^2 \  {\mathbf{B}(\mathrm{PSL}(2,q))} \le  {\mathbf{B}(\mathrm{PSL}(2,q))}^2 + 2  {\mathbf{B}(\mathrm{PSL}(2,q))} +1  
   \end{equation}   
\begin{equation}
 \Longrightarrow  c^2  \ \mathbf{B}(\mathrm{PSL}(2,q))^3 \le   \mathbf{B}(\mathrm{PSL}(2,q))^4 +  2  \mathbf{B}(\mathrm{PSL}(2,q))^3 +    \mathbf{B}(\mathrm{PSL}(2,q))^2 
 \end{equation}   
\begin{equation}
 \Longrightarrow  c^2  \ \mathbf{B}(\mathrm{PSL}(2,q))^3 \le   {\mathbf{B}(\mathrm{PSL}(2,q))}^2  \ {( \mathbf{B}(\mathrm{PSL}(2,q))+1)}^2
 \end{equation}   
and so
\begin{equation}
|\mathrm{PSL}(2,q)|^2  \le  c^2  \ \mathbf{B}(\mathrm{PSL}(2,q))^3 \le   {\mathbf{B}(\mathrm{PSL}(2,q))}^2  \ {( \mathbf{B}(\mathrm{PSL}(2,q)) + 1)}^2   \ \Longrightarrow \ \mathrm{PSL}(2,q) \in \mathcal{H}.
\end{equation}   

Part (ii) is completely proved.

\medskip
\medskip
\medskip
\medskip
\medskip
\medskip
\medskip
\medskip

(iii). Assume that $q=2^{2n+1}$ for some $n \ge 1$.  We note that
\begin{equation}8 \ \mathbf{B}(\mathrm{Sz}(q))^3 = 8 \  \left(\frac{q^4-q^3-(q+1)}{2}\right)^3=  ((q^2+q+2)(q^2-1))^3
\end{equation}
and that
\begin{equation}|\mathrm{Sz}(q)|^2= (q^2(q^2+1)(q-1))^2=q^4(q^2+1)^2(q-1)^2.  \end{equation}
This allows us to conclude that the condition
\begin{equation} ((q^2+q+2)(q^2-1))^3 \ge  q^4(q^2+1)^2(q-1)^2.
\end{equation} 
is always satisfied for all $q \ge 8$ and so
\begin{equation}|\mathrm{Sz}(q)|^2  \le  8  \ \mathbf{B}(\mathrm{Sz}(q))^3,
\end{equation}
from which the result follows. 

Similarly, $\mathrm{Sz}(q)$ possesses a nontrivial partition when $q \ge 8$ by Theorem \ref{niceclassification} and is a simple group when $q \ge 8$. Therefore it remains to check that $\mathrm{Sz}(q) \in \mathcal{H}$ when $q \ge 8$.   This is true, because for all $q \ge 8$ we have by Corollary \ref{l2} and by $c^2 \ge 8 $ that
\begin{equation} 0 \le c^2 -2   \le  {\mathbf{B}(\mathrm{Sz}(q))} + \frac{1}{\mathbf{B}(\mathrm{Sz}(q))}
\end{equation}   
\begin{equation}
\Longrightarrow  \  c^2 \  {\mathbf{B}(\mathrm{Sz}(q))}  - 2  {\mathbf{B}(\mathrm{Sz}(q))} \le  {\mathbf{B}(\mathrm{Sz}(q))}^2  +1  
\end{equation}   
\begin{equation}
   \Longrightarrow  \  c^2 \  \mathbf{B}(\mathrm{Sz}(q)) \le  \mathbf{B}(\mathrm{Sz}(q))^2 + 2   \mathbf{B}(\mathrm{Sz}(q)) +1  
   \end{equation}   
\begin{equation}
 \Longrightarrow  c^2  \ \mathbf{B}(\mathrm{Sz}(q))^3 \le   {\mathbf{B}(\mathrm{Sz}(q))}^4 +   2 { \mathbf{B}(\mathrm{Sz}(q))}^3 +    {\mathbf{B}(\mathrm{Sz}(q))}^2 
 \end{equation}   
\begin{equation}
 \Longrightarrow  c^2  \ \mathbf{B}(\mathrm{Sz}(q))^3 \le   {\mathbf{B}(\mathrm{Sz}(q))}^2  \ {( \mathbf{B}(\mathrm{Sz}(q))+1)}^2
 \end{equation}   
and so
\begin{equation}
|\mathrm{Sz}(q)|^2  \le  c^2  \ \mathbf{B}(\mathrm{Sz}(q))^3 \le   {\mathbf{B}(\mathrm{Sz}(q))}^2  \ {( \mathbf{B}(\mathrm{Sz}(q))+1)}^2   \ \Longrightarrow \ \mathrm{Sz}(q) \in \mathcal{H}.
\end{equation}   
The result (iii) is completely proved.
\medskip
\medskip
\medskip
\medskip
\medskip
\medskip
\medskip
\medskip

    (iv). Assume that $q=p^m$ for $q \ge 5$ and $q=2^r$ for $r \geq 2$. Using the same logic which we have seen in (ii) and (iii) above and involving Corollary \ref{l3}, we get
\begin{equation}
8 \ \mathbf{B}(\mathrm{PGL}(2, q))^3 = 8 \  \left(\frac{q^2+q-2}{2}\right)^3=  (q^2+q-2)^3
\end{equation}
and 
\begin{equation}|\mathrm{PGL}(2, q)|^2= (q(q-1)(q+1))^2=q^2(q+1)^2(q-1)^2.  \end{equation}
This allows us to conclude that the condition
\begin{equation} (q^2+q-2)^3 \ge  q^2(q+1)^2(q-1)^2.
\end{equation} 
is always satisfied for all $q \ge 5$ and so
\begin{equation}|\mathrm{PGL}(2, q)|^2  \le  8  \ \mathbf{B}(\mathrm{PGL}(2, q))^3,
\end{equation}
from which the result follows.  Once more, if $q=2^m$ for $m\geq 2$, then Corollary \ref{l3} implies
\begin{equation} 8 \ \mathbf{B}(\mathrm{PGL}(2, q))^3 = 8 \ \left(\frac{q^2}{2}\right)^3 =q^{6} \end{equation} and \begin{equation}|\mathrm{PGL}(2, q)|^2= (q(q-1)(q+1))^2=q^2(q+1)^2(q-1)^2. \end{equation}
It is evident that \begin{equation} q^{6} \ge q^2(q+1)^2(q-1)^2=q^6-2q^4+q^2. \end{equation} is always satisfied for all $q=2^m$, therefore \begin{equation}|\mathrm{PGL}(2, q)|^2 \le 8 \ \mathbf{B}(\mathrm{PGL}(2, q))^3 \end{equation}
Since $\mathrm{PGL}(2, q)$ possesses a nontrivial partition when $q \ge 4$,  it remains to check that $\mathrm{PGL}(2, q) \in \mathcal{H}$ when $q \ge 4$.   This is true, because for all $q \ge 8$ we may follow a similar argument as done in the previous situations, that is, 
\begin{equation} 0 \le c^2 -2   \le  {\mathbf{B}(\mathrm{PGL}(2,q))} + \frac{1}{\mathbf{B}(\mathrm{PGL}(2, q))}
\end{equation}   
\begin{equation}
\Longrightarrow  \  c^2 \  {\mathbf{B}(\mathrm{PGL}(2,q))}  - 2  {\mathbf{B}(\mathrm{PGL}(2,q))} \le  {\mathbf{B}(\mathrm{PGL}(2,q))}^2  +1  
\end{equation}   
\begin{equation}
   \Longrightarrow  \  c^2 \  {\mathbf{B}(\mathrm{PGL}(2,q))} \le  {\mathbf{B}(\mathrm{PGL}(2,q))}^2 + 2  {\mathbf{B}(\mathrm{PGL}(2,q))} +1  
   \end{equation}   
\begin{equation}
 \Longrightarrow  c^2  \ \mathbf{B}(\mathrm{PGL}(2,q))^3 \le   {\mathbf{B}(\mathrm{PGL}(2,q))}^4 +  2{ \mathbf{B}(\mathrm{PGL}(2,q))}^3 +    {\mathbf{B}(\mathrm{PGL}(2,q))}^2 
 \end{equation}   
\begin{equation}
 \Longrightarrow  c^2  \ \mathbf{B}(\mathrm{PGL}(2,q))^3 \le   {\mathbf{B}(\mathrm{PGL}(2,q))}^2  \ {( \mathbf{B}(\mathrm{PGL}(2,q))+1)}^2
 \end{equation}   
and so
\begin{equation}
|\mathrm{PGL}(2, q)|^2  \le  c^2  \ \mathbf{B}(\mathrm{PGL}(2,q))^3 \le   {\mathbf{B}(\mathrm{PGL}(2,q))}^2  \ {( \mathbf{B}(\mathrm{PGL}(2,q))+1)}^2   \ \Longrightarrow \ \mathrm{PGL}(2, q) \in \mathcal{H}.
\end{equation}   
Part (iv) of the result is completely proved, so the proof ends.
\end{proof}

We end with an useful example, which helps to visualize our main results in concrete situations.

\begin{example}\label{e1}Note that $\mathrm{PGL}(2,3) \simeq S_4$ and  $\mathrm{PGL}(2,4) \simeq A_5$ for which we have already seen that $\mathbf{B}(S_4)=5$ and  $\mathbf{B}(A_5)=8$, moreover $S_4, A_5 \in \mathcal{H}$.

We may use  GAP \cite{gap}, in order to compute the orders of the elements of the group $\mathrm{PGL}(2,2^m)$ for $ 2 \le m \leq 6$, confirming Theorems \ref{l0} and \ref{maintheorem}. First of all,  the maximal cyclic subgroups of $\mathrm{PGL}(2,2^m)$ have order precisely $2$, $2^m-1$, and $2^m+1$.  In particular, if $m=6$, then we consider $\mathrm{PGL}(2,64)$ of order $262080=2^{6} \cdot 3^{2} \cdot 5 \cdot 7 \cdot 13$, $d \in \mathrm{Div}(\mathrm{PGL}(2,64))$, $\pi(\mathrm{PGL}(2,64))=\{2,3,5,7,13\} \subseteq \mathrm{Div}(\mathrm{PGL}(2,64))$ and must observe that $\mathbf{c}_d=|O_d|/\phi(d)$, where
$O_d=\{x \in \mathrm{PGL}(2,64) \mid o(x)=d\}= \{ \mbox{Elements of } \ \mathrm{PGL}(2,64) \ \mbox{of order exactly} \ d\}.$ We find that
\begin{center}
\begin{tabular}{|c|l|c|c|}
            \hline
        $d$ & $\mathbf{c}_d$  & $|L_d(\mathrm{PGL}(2,64))|$ & $\mathbf{b}_d(\mathrm{PGL}(2,64))$ \\
    \hline    
    1  &   1      & 1  &  1 \\
    2  & 4095     & 4096  &  2048 \\
    3  & 2080     & 4161  &  1387 \\
    5  & 2080     & 8065  &  1613 \\
    7  & 2080    & 12481  &  1783 \\
    9  & 2080    & 16641  &  1849 \\
    13 & 2080    & 24193  &  1861 \\
    21 & 2080    & 41601  &  1981 \\
    63 & 2080    & 128961  & 2047 \\
    65 & 2080   & 129025  & 1985 \\
    4  & 0       & 4096  &  1024 \\
    8  & 0       & 4096  &  512 \\
    16 & 0       & 4096  &  256 \\
    32 & 0       & 4096  &  128 \\
    64 & 0       & 4096  &  64 \\
    6  & 0       & 8256  &  1376 \\
    12 & 0       & 8256  &  688 \\
    15 & 0       & 12225 &  815 \\
    18 & 0       & 20736  & 1152 \\
    \vdots &   \vdots      &  \vdots  &   \vdots \\
      262080 & 0    & 262080  &  1 \\
    \hline
    \end{tabular}
\end{center}
\centerline{Fig.1. Values of local breadth for $\mathrm{PGL}
(2,64)$ from \eqref{lb} and \eqref{phib}.}

\medskip

Note that the maximal cyclic subgroups of $\mathrm{PGL}(2,64)$ have order $2$, $63$ and $65$. Moreover $\mathbf{B}(\mathrm{PGL}(2,64))=2048=\mathbf{b}_{2}(\mathrm{PGL}
(2,q))$, which occurs at the order of the maximal subgroup $M \simeq C_2$ of $\mathrm{PGL}(2,64)$. Here we appreciate the value of Theorem \ref{l0}, which easily allows us to conclude that $\mathbf{B}(\mathrm{PGL}(2,64))=2^{2m-1}$ when $m=6$. See also  Corollary \ref{l3}. The bounds of Theorem \ref{maintheorem} (iv) can be easily checked  for $\mathrm{PGL}(2,64)$.
\end{example}

We give another evidence of the computational advantages of our main results.

\begin{example}\label{e2}
We focus on  $\mathrm{PSL}(2,9) \simeq A_6$, again using GAP \cite{gap}. Here we have a  simple group of nontrivial partition of  $|\mathrm{PSL}(2,9)|=360$ with $\pi(\mathrm{PSL}(2,9)) = \{2,3,5\}$.

With the same notations which have been used in the Example \ref{e1}, we find for $d=6$ and $q=9$ that $\mathbf{b}_{6}(\mathrm{PSL}(2,9))< \mathbf{b}_{2}(\mathrm{PSL}(2,9))$ and $\mathbf{b}_{6}(\mathrm{PSL}(2,9))< \mathbf{b}_{3}(\mathrm{PSL}(2,9))$ even though $|L_6(\mathrm{PSL}(2,9))| \geq |L_2(\mathrm{PSL}(2,9))|$ and $|L_6(\mathrm{PSL}(2,9))| \geq |L_3(\mathrm{PSL}(2,9))|$.

In the present situation,  the maximal cyclic subgroups of $\mathrm{PSL}(2,9)$ are of order $3$, $4$ and $5$,  and $\mathbf{B}(\mathrm{PSL}(2,9))=34=\frac{ 9 \cdot (9-1)}{2}-2=\mathbf{b}_{4}(\mathrm{PSL}(2,9))$, which can be calculated using Corollary \ref{l1} of \eqref{pslodd}  and it occurred at the order of the maximal subgroup $M\simeq C_4=C_{\frac{9-1}{2}}$ of $\mathrm{PSL}(2,9)$. 
\begin{center}
\begin{tabular}{|c|l|c|c|}
            \hline
        $d$ & $\mathbf{c}_d$  & $|L_d(\mathrm{PSL}(2,9))|$ & $\mathbf{b}_d(\mathrm{PSL}(2,9))$ \\
    \hline 
    1  &   1    & 1    &  1 \\  
    2  & 45     & 46   &  23 \\
    3  & 40     & 81   &  27 \\
    4  & 45     & 136  &  34 \\
    5  & 36     & 145  &  29 \\
    6  & 0    & 126  &  21 \\
    8 & 0     & 136  &  17 \\
    9 & 0     & 81  &  9 \\
    10 & 0    & 190  & 19 \\
    12 & 0    & 216  & 18 \\
    15  & 0       & 225  &  15 \\
    18  & 0       & 126  &  7 \\
    20 & 0       &  280 & 14  \\
    24 & 0       &  216 &  9 \\
    30 & 0       &  270 &  9 \\
    36  & 0       &  216 & 6 \\
    40 & 0       & 280  &  7 \\
    60 & 0       & 360 &  6 \\ 
    72 & 0       & 216  & 3 \\ 
    90 & 0       & 270  &  3 \\
    120 & 0       & 360 &  3 \\
    180 & 0       & 360  & 2 \\
    360 & 0       & 360  & 1 \\
   \hline
    \end{tabular}
\end{center}
\medskip
\centerline{Fig.2. Values of local breadth for $\mathrm{PSL}
(2,9) \simeq A_6$ from \eqref{lb} and \eqref{phib}.}

\medskip

It is now easy to check that the bound of Theorem \ref{maintheorem} (ii) is satisfied.
\end{example}

\medskip

The determination of  $p$-groups  $G$ with  $H_p (G) \neq G$ in $\mathcal{H}$ (or not) becomes more difficult to check. Definitively, Theorem \ref{l0} helps since it works for all groups with nontrivial partitions, but the presence of bounds as per Theorem \ref{maintheorem} may be more challenging. This happens for the  groups that we mentioned in (ii), (iii) and (iv) of Theorem \ref{niceclassification}. We do not know whether they are all in $\mathcal{H}$, or not. In fact the following problem is open:

\begin{conjecture} The class of groups $\mathcal{H}$ might contain all  groups in Theorem \ref{niceclassification}. It might be possible that $\mathcal{H}$ contains also  simple groups and nonsimple groups, which do not possess necessarily   a nontrivial partition. 
\end{conjecture}

\medskip
\medskip
\medskip



\end{document}